\documentclass[a4paper,draft,reqno,12pt]{amsart}
\usepackage[english]{babel}
\usepackage{amsmath}
\usepackage{amssymb}
\usepackage{amscd}
\usepackage{amsthm}
\usepackage{euscript}

\usepackage{mathrsfs}

\usepackage{graphicx}

\newtheorem{lemma}{Lemma}
\newtheorem{theorem}{Theorem}

\newtheorem{corollary}{Corollary}
\theoremstyle{definition}
\newtheorem{definition}{Definition}
\newtheorem{ex}{Example}
\theoremstyle{remark}
\newtheorem {re}{Remark}

\DeclareMathOperator{\Aut}{Aut}
\DeclareMathOperator{\SAut}{SAut}

\DeclareMathOperator{\GL}{GL}

\def\Ker{{\rm Ker}\,}

\def\CC{{\mathbb C}}
\def\KK{{\mathbb K}}

\def\ZZ{{\mathbb Z}}

\def\NN{{\mathbb N}}
\def\QQ{{\mathbb Q}}

\begin{document}

\date{}

\author{Ilya Boldyrev}
\address{Lomonosov Moscow State University, Faculty of Mechanics and Mathematics, Department of Higher Algebra, Leninskie Gory 1, Moscow, 119991 Russia; }\email{anesecer@gmail.com}

\author{Sergey Gaifullin}
\address{Lomonosov Moscow State University, Faculty of Mechanics and Mathematics, Department of Higher Algebra, Leninskie Gory 1, Moscow, 119991 Russia; \linebreak and \linebreak
National Research University Higher School of Economics, Faculty of Computer Science, Pokrovsky Boulevard 11, Moscow, 109028, Russia}
\email{sgayf@yandex.ru}

\title{Automorphisms of nonnormal toric varieties}

\markboth{Ilya Boldyrev, Sergey Gaifullin}{Automorphisms of nonnormal toric varieties}



\thanks{The second author was supported by RSF grant 19-11-00172.}

\maketitle

\begin{abstract}
In this paper we prove criteria for a nonnormal toric variety to be flexible, to be rigid and to be almost rigid. For rigid and almost rigid toric varieties we describe the automorphism group explicitly.
\end{abstract}


\section{Introduction}
Let $\KK$ be an algebraically closed field of characteristic zero.We denote the additive group of the field $\KK$ by $\mathbb{G}_a$. Let us consider an irreducible affine algebraic variety $X$ over $\KK$. If an algebraic torus $T\cong(\KK^\times)^n$ acts on $X$ with an open orbit, then $X$ is called {\it toric}. Often by toric  variety one mean a normal variety. But we do not suppose $X$ to be normal. 

We want to investigate the automorphism group of $\mathrm{Aut}(X)$. In \cite{AFKKZ} the subgroup $\SAut(X)\subset\Aut(X)$ of special automorphisms  is defined. By definition the subgroup $\SAut(X)$ is generated by all algebraic subgroups isomorphic to $\mathbb{G}_a$. A point $x\in X$ is called {\it flexible}, if the tangent space $T_xX$ is generated by tangent vectors to orbits of $\mathbb{G}_a$-subgroups. Flexibility of $x$ is equivalent to openness of the orbit $\SAut(X)x$  in $X$. If all regular points on $X$ are flexible, the variety $X$ is called {\it flexible}. Recall that an action of a group $G$ on a set is called {\it infinitely transitive}, if for each two finite collectns $(a_1,\ldots,a_m), a_i\neq a_j$  and $(b_1, \ldots, b_m), b_i\neq b_j$ with coinciding number of elements there exists an element $g\in G$ such that $g\cdot a_i=b_i$ for all $1\leq i\leq m$.
The interest in flexible varieties is largely due to the following theorem.

{\bf Theorem. \cite[Theorem 0.1]{AFKKZ}}   {\it Let $X$ be an irreducible affine variety of dimension $\geq 2$. Then the following conditions are equivalent. 

(i) The group $\SAut(X)$  acts transitively on $X^{reg}$.

(ii) The group $\SAut(X)$  acts infinitely transitively on $X^{reg}$. 

(iii) The variety $X$ is flexible.
}

There are a lot of works proving flexibility of some classes of varieties, for example \cite{AFKKZ}, \cite{AKZ}, \cite{Pe}, \cite{Sh}, \cite{GSh}. One of the first example of flexible varieties is nondegenerate normal toric varieties, see~\cite{AKZ}. Recall that a toric variety is called nondegenerate if it does not admit nonconstant invertible regular functions. If we remove the condition of normality, a toric variety can be nonflexible. For example, the toric curve $\{x^2=y^3\}$ is not flexible. The main result of this paper is a criterium of flexibility for not necessary normal toric varieties. The answer is given both in combinatorial terms, see Theorem~\ref{gib} and Corollaries~\ref{vt} and~\ref{dy}, and  in geometrical terms, see Corollary~\ref{geo}. Note that Theorem~\ref{gib} is a generalization of the rezult~\cite[Theorem~2.1]{AKZ}. Indeed, in case of normal toric variety Theorem~\ref{gib} gives the condition that the variety is nondegenerate.

For flexible varieties the group $\SAut(X)$ is big and it acts on $X^{reg}$ transitively. In some sense  the opposite situation occurs when the group $\SAut(X)$ is trivial. In this case $X$ is called {\it rigid}. If $X$ is a rigid variety, then there exists a unique maximal algebraic torus in the automorphism group $\Aut(X)$. Sometimes it is possible to describe the automorphism group of $X$ explicitly, see \cite{AG}, \cite{Ga}. In this paper we prove a criterium of rigidity for toric varieties, see Theorem~\ref{zhes}. And we obtain an explicit description of automorphism groups of rigid toric varieties, see Theorem~\ref{autzhes}. The automorphism group of a rigid toric variety $X$ is isomorphic to a semidirect product of the torus $T$ and a discrete subgroup $S(X)$, which we call {\it the group of symmetries of weight monoid} of $X$, see Definition~\ref{ops}. If the variety $X$ does not admit nonconstant invertible functions, then the subgroup $S(X)$ is finite. So the automorphism group of a rigid toric variety without nonconstant invertible functions is a finite extension of the torus $T$. 

We describe automorphism group for one more class of toric varieties. This is the class of almost rigid toric varieties.
A variety $X$ is called {\it almost rigid}, if it is not rigid and the subgroup $\SAut(X)$ of special automorphism is commutative. In this paper we prove a criterium for a toric variety to be almost rigid, see Theorem~\ref{pz}. Also we describe the automorphism group of an almost rigid toric variety as a semidirect product of four subgroups, see Theorem~\ref{apz}. 

Authors are grateful to Ivan Arzhantsev for useful discussions.  

\section{Locally nilpotent derivations}

Algebraic $\mathbb{G}_a$-subgroups of $\Aut(X)$ are in correspondence with locally nilpotent derivations (LND) of the algebra of regular functions~$\KK[X]$. Recall that a linear mapping $\delta\colon\KK[X]\rightarrow\KK[X]$ is called a {\it derivation} of $\KK[X]$, if it satisfies the Leibnits rule  $\delta(fg)=f\delta(g)+\delta(f)g$. A derivation~$\delta$ is called {\it locally nilpotent}, if for each $f\in\KK[X]$ there exists a positive integer number $n$ such, that $\delta^n(f)=0$. More information about LNDs one can find for example in the book \cite{Fr}. If we fix an LND~$\delta$ we can consider the following operator on $\KK[X]$, which we call the {\it exponent} of $\delta$:
$$
\exp(\delta)(f)=f+\delta(f)+\frac{\delta^2(f)}{2!}+\frac{\delta^3(f)}{3!}+\ldots
$$
Since $\delta$ is locally nilpotent, this sum is finite. The operator $\exp(\delta)$ gives an automorphism of $\KK[X]$ and hence an automorphism of the variety $X$. If a function $f$  belonged to the kernel of an LND $\delta$, then the mapping~$f\delta$ is also an LND. It is called a {\it replica} of the derivation~$\delta$. Each LND~$\delta$ corresponds to a $\mathbb{G}_a$-subgroup 
$$\mathcal{H}_\delta=\left\{\exp(t\delta)\mid t\in\KK\right\}\subset\Aut(\KK[X]).$$
Moreover, each algebraic $\mathbb{G}_a$-subgroup corresponds to an LND.

Let us fix a grading of $\KK[X]$  by a commutative group $G$: 
$$\KK[X]=\bigoplus_{g\in G}\KK[X]_g.$$ 
Then a derivation $\delta$ of $\KK[X]$ is called $G$-homogeneous of degree $g_0$, if for every $g\in G$ and a homogeneous element $f\in \KK[X]_g$ we have $\delta(f)\in\KK[X]_{g+g_0}$. If $X$ is a toric variety, then the algebra $\KK[X]$ admits a natural grading  of the group of characters of the torus $\mathfrak{X}(T)$. In Section~\ref{lnd} we define $\mathfrak{X}(T)$-homogeneous locally nilpotent derivations, corresponding to so called {\it Demazure roots}. Further this derivations play a key role. If we fix a $\mathbb{Z}$-grading of $\KK[X]$, then every LND $\partial$ can be decomposed onto a sum of homogeneous derivations $\partial=\sum_{i=l}^k\partial_i$. Homogeneous components $\partial_l$ and $\partial_k$ in this sum with the minimal and the maximal degrees  are locally nilpotent, see~\cite[Section~3]{FZ}. If we have $\mathbb{Z}^n$-grading, then each LND can be decomposed onto a sum of homogeneous derivations. The convex hull of degrees of summands is a polyhedron. Derivations corresponding to its vertices are locally nilpotent. 

We need the following lemma. 

\begin{lemma}\label{pzh}
The group $\SAut(X)$ is commutative if and only if for all locally nilpotent derivations of the algebra $\KK[X]$ kernels coincide.
\end{lemma}
\begin{proof}
Suppose, there are two LNDs $\delta_1$ and $\delta_2$ of the algebra $\KK[X]$ such that their kernels do not coincide. We can assume, that there exists $f\in\Ker(\delta_1)\setminus\Ker(\delta_2)$. Let us prove that there are two noncommuting LNDs. If $\delta_1\circ\delta_2\neq\delta_2\circ\delta_1$, we done. Otherwise let us consider following two LNDs $f\delta_1$ and $\delta_2$. 
We have
$$(\delta_2\circ f\delta_1)(g)=\delta_2(f)\delta_1(g)+f(\delta_2\circ\delta_1)(g)=\delta_2(f)\delta_1(g)+(f\delta_1\circ\delta_2)(g).$$
If we have such $g$, that $\delta_1(g)\neq 0$, then $(\delta_2\circ f\delta_1)(g)\neq (f\delta_1\circ\delta_2)(g)$. So we obtain two noncommuting LNDs. It is easy to see, that groups $\mathcal{H}_\delta$ corresponding to these LNDs are not commute. Therefore, $\SAut(X)$ is not commutative.

Now let all kernels of LNDs coincide. Let us consider two LNDs $\delta_1$ and $\delta_2$. By \cite[Principle~12]{Fr} there exist such elements $f,g\in \Ker(\delta_1)=\Ker(\delta_2)$, that $f\delta_1=g\delta_2$.  Therefore $\delta_1$ and $\delta_2$ commute. Hence, $\mathcal{H}_{\delta_1}$ and $\mathcal{H}_{\delta_2}$ commute. So, $\SAut(X)$ is generated by a set of pairwise commuting commutative subgroups. Therefore, $\SAut(X)$ is commutative.
\end{proof}

\section{Toric varieties}

In this section we give basic facts about toric varieties. More information about toric varieties one can find in books~\cite{CLSch} and~\cite{Ful}. An irreducible algebraic variety is called {\it toric}, if an algebraic torus $T=(\mathbb{K}^\times)^n$ algebraically acts on it with an open orbit. We can assume the action of $T$ on $X$ to be effective. Note that we do not assume toric variety to be normal. Let $X$ be affine. An affine toric variety $X$ corresponds to a finitely generated monoid $P$ of weights of $T$-semiinvariant regular functions. Let us identify the group of characters $\mathfrak{X}(T)$ with a free abelian group $M=\mathbb{Z}^n$. A vector $m\in\mathbb{Z}^n$ with integer coordinates corresponds to the character~$\chi^m$. Since the open orbit on $X$ is isomorphic to $T$, we have an embedding of algebras of regular functions $\mathbb{K}[X]\hookrightarrow\mathbb{K}[T]$. Identifying the algebra $\mathbb{K}[X]$ with its image we obtain the following subalgebra graded by $P$
$$
\mathbb{K}[X]=\bigoplus_{m\in P}\mathbb{K}\chi^m\subset \bigoplus_{m\in M}\mathbb{K}\chi^m=\mathbb{K}[T].
$$
Let us consider the vector space $M_{\mathbb{Q}}=M\otimes_{\mathbb{Z}}\mathbb{Q}$ over the field of rational numbers. The monoid $P$ generates the cone $\sigma^{\vee}=\mathbb{Q}_{\geq0}P\subset M_{\mathbb{Q}}$. Since $P$ is finitely generated, the cone $\sigma^\vee$ is a finitely generated polyhedral cone. Since the action of~$T$ on $X$ is effective, the cone $\sigma^\vee$ does not belong to any proper subspace of $M_\QQ$. The variety $X$ is normal if and only if the monoid $P$ is {\it saturated}, i.e. $P=M\cap \sigma^{\vee}$. If $P$ is saturated, then the monoid $P_{sat}=M\cap \sigma^{\vee}$ we call the {\it saturation}  of the monoid $P$. Elements of $P_{sat}\setminus P$ we call {\it holes} of $P$. Let us give some definition according to \cite{TY}.

\begin{definition}
An element $p$ of the monoid $P$ is called {\it saturation point} of $P$, if the moved cone $p+\sigma^{\vee}$ has no holes, i.e. $(p+\sigma^{\vee})\cap M\subset P$. 
 
A face $\tau$ of the cone $\sigma^{\vee}$ is called {\it almost  saturated}, if there is a saturation point of  $P$ in $\tau$. 
Otherwise $\tau$ is called a {\it nowhere saturated} face.
\end{definition}

The following lemma is known, but for the reader's  convenience we give a proof.
\begin{lemma}\label{mg}
The maximal face, i.e. the whole cone $\sigma^\vee$, is almost saturated.
\end{lemma}
\begin{proof}
Let $a_1, \ldots, a_r$ be a system of generators of a monoid $P$. Since the set
$$S=\left\{\lambda_1 a_1+\ldots+\lambda_r a_r\mid \lambda_i\in\QQ, 0\leq \lambda_i<1 \right\}$$
is bounded, it contains only finite number of elements of $M$. Denote 
$
S\cap M=\{b_1,\ldots, b_l\}.
$
Since the action of $T$ on $X$ is effective, the group generated by $P$ coincides with $M$. Hence, for every~$j$ there exist $c_j$ and $d_j$ in $P$ such, that $b_j=c_j-d_j$. Let us consider $d=d_1+\ldots+d_l$. 
Let us prove, that 
$$(d+\sigma^\vee)\cap M\subset P.$$ 
Suppose $d+m\in(d+\sigma^\vee)\cap M$. Since vectors $a_1,\ldots, a_r$ generate $\sigma^\vee$, the element~$m$ is a linear combination $$m=\mu_1a_1+\ldots+\mu_r a_r,\ \mu_i\in\QQ,\ \mu_i\geq 0.$$
We denote by $[\ \cdot\ ]$ the integer part of a number, and by $\{\,\cdot\,\}$ the fractional part of a number. We obtain
$$
m=\sum_{i=1}^r \left[\mu_i\right]a_i+\sum_{i=1}^r \{\mu_i\}a_i,
$$
The first summand is in $M$ and $m\in M$. Therefore, the second summand is in~$M$. Since $\{\mu_i\}<1$, we obtain
$\sum_{i=1}^r \{\mu_i\}a_i\in S$. Hence, there is such~$k$ that $\sum_{i=1}^r\{\mu_i\}a_i=b_k$. We obtain
$$
d+m=\sum_{j=1}^ld_j+\sum_{i=1}^r \left[\mu_i\right]a_i+b_k=\sum_{j\neq k}d_j+\sum_{i=1}^r \left[\mu_i\right]a_i+c_k\in P.
$$
\end{proof}

The lattice of one-parameter subgroups of the torus $T$ we denote by $N$. The lattice $N$ is a dual lattice to $M$. There is a natural pairing
$M\times N\rightarrow \mathbb{Z}$, which we denote $\langle \cdot,\cdot\rangle$. This pairing can be extended to a pairing between vector spaces $N_\QQ=N\otimes_\ZZ\QQ$ and $M_\QQ$. In the space $N_\QQ$ we define the cone $\sigma$ dual to $\sigma^\vee$, by the rule
$$
\sigma=\{v\in N_\QQ\mid\forall w\in\sigma^\vee : \langle w,v\rangle\geq 0\}.
$$
The finitely generated polyhedral cone $\sigma$ is {\it pointed}, i.e. it does not contain any nontrivial subspaces.

There is a bijection between $k$-dimensional faces of $\sigma$ and $(n-k)$-dimensional faces of~$\sigma^\vee$. A face $\tau\preccurlyeq\sigma$ corresponds to the face  $\widehat{\tau}=\tau^{\bot}\cap\sigma^\vee\preccurlyeq \sigma^\vee$. Also there is a bijection between $(n-k)$-dimensional faces of $\sigma^\vee$ and $k$-dimensional $T$-orbits on $X$. A face $\widehat{\tau}\preccurlyeq \sigma^\vee$ corresponds to the orbit, which is open in the set of zeros of the ideal 
$$I_{\widehat{\tau}}=\bigoplus_{m\in P\setminus\widehat{\tau}}\KK\chi^m.$$
The composition of these bijections gives a bijection between $k$-demensional faces of the cone $\sigma$ and $k$-dimensional $T$-orbits. The orbit corresponding to a face $\tau$, we denote by~$O_\tau$.

If $X$ is normal, then the singular locus $X^{sing}$ has codimension not less than~$2$. Hence, the algebras of regular functions on $X$ and on the regular locus $X^{reg}$ coincide. In the case of nonnormal variety this algebras can be different. Despite the variety $X^{reg}$ is not affine, it is convenient to come from $X$ to $X^{reg}$, since the algebra $\KK[X^{reg}]$ is easier than $\KK[X]$.

\begin{lemma}\label{co}
Let $X$ be a toric variety and $\sigma$ be the corresponding cone. Let us remove such extremal rays of  $\sigma$, that correspond to orbits consisting of singular points. Consider the cone $\gamma$, generated by other extremal rays. Then the algebra of regular functions on quasiaffine variety $X^{reg}$ has the following view
$$\KK[X^{reg}]=\bigoplus_{m\in M\cap \gamma^\vee}\KK\chi^m.$$
\end{lemma}
\begin{proof}
Since the variety $X^{reg}$ is $T$-invariant, the algebra $\KK[X^{reg}]$ can be decomposed onto the direct sum of weight components. Therefore, we need to understand, which functions $\chi^m$ are regular on $X^{reg}$. Since $X^{reg}$ is a normal variety, the algebra $\KK[X^{reg}]$ is a direct sum of components corresponding to all integer points in some cone. Each function $\chi^m, m\in M$ is regular on the open orbit. It is known that the order of zero or pole of $\chi^m$ on the divisor $\overline{O_\rho}$ equals $\langle m,n_\rho\rangle$. We obtain the condition $\langle m,n_\rho\rangle\geq 0$ for all $\rho$ such, that $O_\rho$ consists of regular points. These conditions give the cone $\gamma^\vee$. 
\end{proof}

Sometimes the monoid $P$ has some symmetries. Let us consider the following discrete subgroup of the automorphism group of a toric variety. 
Let $\varphi\colon M_\QQ\rightarrow M_\QQ$ be a linear operator such that $\varphi\in\mathrm{GL}_n(\ZZ)$, i.e. $\varphi(M)=M$. Suppose that $\varphi(P)=P$. The operator $\varphi$ induces the automorphism $\alpha(\varphi)$ of the algebra $\KK[X]$ by the following formula
$
\alpha(\varphi)(\chi^m)=\chi^{\varphi(m)}.
$
It is easy to see, that $\alpha(\varphi\circ\psi)=\alpha(\varphi)\circ\alpha(\psi)$.  Therefore, $\alpha$ is an injection from the stabilizer 
$St_{\mathrm{GL}_n(\ZZ)}(P)$ to $\Aut(X)$.
\begin{definition}\label{ops}
The image of the injection $\alpha$ we call {\it the symmetry group of the monoid $P$}. We denote it by $S(X)$. 
\end{definition}

\begin{re}\label{zam}
Assume that the cone $\sigma^\vee$ does not contain any lines. This is equivalent to the fact that the cone $\sigma$ is not contained in any hyperspace, and to the fact that $X$ does not admit any nonconstant invertible functions. 
Then the group $S(X)$  is finite. Indeed the operator $\varphi$ permute primitive vectors on extremal rays of the cone $\sigma^{\vee}$. Also the group $S(X)$ is finite in the case, when $\sigma^{\vee}$ is a line. In this case $S(X)$ is isomorphic to a subgroup in $\mathrm{GL}_1(\ZZ)\cong \mathbb{Z}_2$. 

In all other cases $S(X)$ is an infinite discrete group. Indeed there are two variants. The first case, $\sigma^\vee$ is a subspace $M_\QQ$ of dimension $n\geq 2$. In this case it is easy to see that $P=M$ and $S(X)=\mathrm{GL}_n(\ZZ)$. The second case is  that $\sigma^\vee$ contains a nontrivial subspace, but it is not a space. Denote by $W$ the maximal subspace in $\sigma^\vee$. Then $L=W\cap M$ is a nonzero sublattice of $M$. Let us fix a vector $v\in N$ belonging to a relative interior of $\sigma$. Then for every element $l\in L$ we define an operator $\varphi_l\colon M_{\QQ}\rightarrow M_{\QQ}$ by the formula $\varphi_l(m)=m+\langle m,v\rangle l$. We have $\varphi_l(P)=P$. Automorphisms $\alpha(\varphi_l)$ for all $l\in L$ form a subgroup in $\Aut(X)$, isomorphic to $L$. 
\end{re}

\section{LND on toric varieties}\label{lnd}

In the paper \cite{D} all $M$-homogeneous LNDs  on normal toric varieties are described. Consider an extremal ray $\rho$ of $\sigma$. Let us denote by $n_\rho$ the primitive integer vector on $\rho$. The element $e\in M$ is called a {\it Demazure root, corresponding to $\rho$}, if $\langle e, n_\rho \rangle=-1$ and for each other extremal ray $\xi$ of the cone $\sigma$ it is true $\langle e,n_{\xi}\rangle\geq 0$.
The set of all Demazure roots corresponding to an extremal ray~$\rho$ we denote ${\mathscr R}_\rho$. It is easy to see, that for each extremal ray $\rho$ the set
${\mathscr R}_\rho$ is nonempty. Moreover it is infinite.
Each Demazure root $e\in{\mathscr R}_\rho$ corresponds to the LND $\partial_e$ of the algebra $A=\bigoplus_{m\in P_{sat}}\KK\chi^m$ defined on homogeneous elements by the formula
$$
\partial_e(\chi^m)=\langle m,n_\rho\rangle \chi^{m+e}.
$$
The derivation $\partial_e$ is $M$-homogeneous of degree $e$. The kernel of the derivation $\partial_e$ is the subalgebra $\bigoplus_{m\in M\cap \widehat{\rho}}\KK\chi^m$.
If the subalgebra $\KK[X]\subset A$ is $\partial_e$-invariant, then $\partial_e$ induces an LND of the algebra $\KK[X]$, which we denote $\delta_e$. It is easy to see, that the subalgebra $\KK[X]$ is $\partial_e$-invariant if and only if  $(P+e)\cap P_{sat}\subset P$. 
\begin{ex}
Let us consider the variety corresponding to the monoid
$$P=\{(a,b)\in\ZZ^2\mid a\geq0, b\geq 0, (a,b)\neq (1,0)\}.$$
We have $\sigma^\vee=\mathrm{cone}((1,0),(0,1))$, $\sigma=\mathrm{cone}((1,0),(0,1))$. Let us denote $\rho=\QQ_{\geq 0}(0,1)$. Then 
${\mathscr R}_\rho=\{e_k=(k,-1)\mid k\in\ZZ_{\geq 0}\}.$
The derivation $\delta_{e_k}$ is well defined if $k\geq 2$.

$$
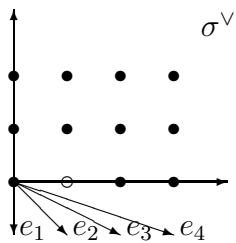
\begin{picture}(100,80)
\put(70,75){$\sigma^\vee$}
\put(0,20){\vector(1,0){80}}
\put(0,20){\vector(0,1){65}}
\put(0,20){\circle*{4}}
\put(0,40){\circle*{4}}
\put(0,60){\circle*{4}}
\put(20,20){\circle{4}}
\put(20,40){\circle*{4}}
\put(20,60){\circle*{4}}
\put(40,20){\circle*{4}}
\put(40,40){\circle*{4}}
\put(40,60){\circle*{4}}
\put(60,20){\circle*{4}}
\put(60,40){\circle*{4}}
\put(60,60){\circle*{4}}
\put(0,20){\vector(0,-1){20}}
\put(2,0){$e_1$}
\put(0,20){\vector(1,-1){20}}
\put(22,0){$e_2$}
\put(0,20){\vector(2,-1){40}}
\put(42,0){$e_3$}
\put(0,20){\vector(3,-1){60}}
\put(62,0){$e_4$}
\end{picture}
$$

\end{ex}

The following lemma is a key lemma. It gives a criterium of existing a well defined LND $\delta\in{\mathscr R}_\rho$ for a given $\rho$. The answer is given in terms of smoothness of orbits of codimension 1 and in terms of the monoid~$P$.  

\begin{lemma}\label{eu}
Let $X$ be a toric variety, $\sigma$ be the corresponding cone and $\rho$ be an extremal ray of~$\sigma$. Denote by $O_{\rho}$ the corresponding orbit. Then the following conditions are equivalent:

1) the face $\widehat{\rho}$ of the cone $\sigma^{\vee}$ is almost saturated;

2) there exist a Demazure root $e\in {\mathscr R}_\rho$ of the cone $\sigma$ such, that the corresponding derivation $\delta_e$ of the algebra $\mathbb{K}[X]$ is well defined;

3) the orbit $O_{\rho}$ consists of smooth points.
\end{lemma}
\begin{proof}
$1\Rightarrow 2$ 
Let $P$ be the weight monoid of the variety $X$.
Since the face $\widehat{\rho}$ is almost saturated, there is an element $w\in P\cap\widehat{\rho}$ such, that 
$$
\left(w+\sigma^{\vee}\right)\cap M\subset P.
$$
Let $r\in{\mathscr R}_\rho$ be a Demazure root corresponding to the extremal ray $\rho$. Let us take an integer vector $v$ in the relative interior of the face $\widehat{\rho}$. Then for every $k\in\mathbb{N}$ and for every extremal ray $\xi\neq\rho$ of the cone $\sigma$ we have
$$\langle r+kv, n_{\rho}\rangle=\langle r, n_{\rho}\rangle+k\langle v, n_{\rho}\rangle=-1;$$
$$\langle r+kv, n_{\xi}\rangle=\langle r, n_{\xi}\rangle+k\langle v, n_{\xi}\rangle\geq 0.$$
Hence, $r+kv\in{\mathscr R}_\rho$  for each positive integer $k$. It is easy to see that for sufficiently large~$k$ we have  $(r+kv+\sigma^{\vee})\cap\sigma^\vee\subset w+\sigma^{\vee}$. Let us denote $e=r+kv$ for this $k$. Then
$$
(e+P)\cap P_{sat}\subset (e+\sigma^{\vee})\cap\sigma^{\vee}\subset w+\sigma^{\vee}.
$$
So $(e+P)\cap P_{sat}\subset (w+\sigma^{\vee})\cap M\subset P$.
Therefore, the formula $\delta_e(\chi^m)=\langle m,n_\rho \rangle \chi^{m+e}$
gives a well defined LND of $\mathbb{K}[X]$.

$2\Rightarrow 3$ The proof of this implication is analogical to the proof of Lemma~14 and Theorem~6 in \cite{Sh}. The derivation $\delta_e$ corresponds to the $\mathbb{G}_a$-action $\mathcal{H}_e$ on $X$. Since the ideal of functions vanishing on $O_\rho$ has the view
$$\mathrm{I}_{\widehat{\rho}}=\bigoplus_{m\in P\setminus \widehat{\rho}}\KK\chi^m\ \ \text{and}\ \ \mathrm{Ker}\,\delta_e=\bigoplus_{m\in P\cap \widehat{\rho}}\KK\chi^m,$$
there exists a function $f=\chi^m$ such, that $f\in \mathrm{I}_{\widehat{\rho}}$, but $\delta_e(f)\notin \mathrm{I}_{\widehat{\rho}}$. Let us fix such a point $x\in O_\rho$, that $\delta_e(f)(x)\neq 0$. We can assume that $x$ does not belong to closures of orbits $\overline{O_{\xi}}$ for all extremal rays $\xi\neq\rho$ of the cone $\sigma$. Then there exists $s\in\mathcal{H}_e$ such that $s\cdot x\notin \overline{O_\rho}$. The $\mathcal{H}_e$-orbit of $x$ is  irreducible, and it is not contained in any $ \overline{O_{\zeta}}$. Therefore there exists a point $y=s\cdot x$ such that $y$ does not belong to closures of the orbits $ \overline{O_{\zeta}}$ for all extremal rays $\zeta$ of the cone $\sigma$. Hence, $y$ belonged to the open $T$-orbit. Since the open $T$-orbit consists of smooth points, the point $y$ is smooth. Therefore, the point $x$ is smooth. That is the orbit $O_\rho$ consists of smooth points.

$3\Rightarrow 1$ Let us fix a point $x\in O_\rho$. The point $x$ is a smooth point of the closure~$\overline{O_\rho}$. The algebra of regular functions on the variety
$\overline{O_\rho}$ has the view $\KK[\overline{O_\rho}]=\bigoplus_{m\in P\cap\widehat{\rho}}\KK\chi^m$. Let us denote by $\mathfrak{m}_x$ the maximal ideal in the algebra $\KK[\overline{O_\rho}]$ corresponding to $x$. And we denote by  $\mathfrak{M}_x$ the maximal ideal in the algebra $\KK[X]$ corresponding to $x$. 
Then $\mathfrak{M}_x=\mathfrak{m}_x\oplus \mathrm{I}_{\widehat{\rho}}$. Hence, 
$$
\mathfrak{M}_x/\mathfrak{M}_x^2=\mathfrak{m}_x/\mathfrak{m}_x^2\oplus\mathrm{I}_{\widehat{\rho}}/\left(\mathrm{I}_{\widehat{\rho}}^2+\mathfrak{m}_x\mathrm{I}_{\widehat{\rho}}\right).
$$
Since $x$ is a smooth point of $X$, we have
$$\dim X=\dim \mathrm{T}_xX=\dim \mathfrak{M}_x/\mathfrak{M}_x^2.$$ From the other hand, since $x$ is a smooth point of $\overline{O_\rho}$, we have 
$$\dim \overline{O_\rho}=\dim \mathrm{T}_x\overline{O_\rho}=\dim \mathfrak{m}_x/\mathfrak{m}_x^2.$$ 
The codimension of the orbit $O_\rho$ equals 1. This implies  
$$
\dim\mathrm{I}_{\widehat{\rho}}/\left(\mathrm{I}_{\widehat{\rho}}^2+\mathfrak{m}_x\mathrm{I}_{\widehat{\rho}}\right)=1.
$$
Let us choose $\chi^w\in\mathrm{I}_{\widehat{\rho}}$ such, that $\mathrm{I}_{\widehat{\rho}}=\langle \chi^w\rangle\oplus(\mathrm{I}_{\widehat{\rho}}^2+\mathfrak{m}_x\mathrm{I}_{\widehat{\rho}})$.

Let us fix a system of generators $a_1,\ldots, a_r$  of the monoid $P$. Assume that $a_1,\ldots,a_l$ belong to the face $\widehat{\rho}$, and $a_{l+1},\ldots, a_r$ do not belong to this face. Suppose the lattice $L=\mathbb{Z}(a_1,\ldots,a_l)$ does not coincide with $M\cap \rho^{\bot}$. Let us consider the lattice $\Lambda=\ZZ(a_1,\ldots,a_l, w)$. Let us prove that $\mathrm{I}_{\widehat{\rho}}$ is contained in the subspace $W=\bigoplus_{\lambda\in\Lambda}\KK\chi^\lambda$. For each positive integer $k$ we denote 
$$V_k=\bigoplus_{m\in M, \langle m,n_\rho\rangle=k}\KK\chi^m.$$
Let us prove by induction that $\left(\mathrm{I}_{\widehat{\rho}}\cap V_k\right)\subset W$. 

Base of induction, let $k_0$ be the minimal number such that $V_{k_0}\cap\mathrm{I}_{\widehat{\rho}}\neq\{0\}$. Since $\mathrm{I}_{\widehat{\rho}}^2\cap V_{k_0}=\{0\}$, we obtain $\mathrm{I}_{\widehat{\rho}}\cap V_{k_0}\subset \langle \chi^w\rangle\oplus\mathfrak{m}_x\mathrm{I}_{\widehat{\rho}}$. If  $\mathrm{I}_{\widehat{\rho}}\cap V_{k_0}$ is not contained in $\bigoplus_{m\in w+L}\KK\chi^m$, then there exists an $\alpha\in M\setminus \Lambda$ such that $\langle \alpha,n_\rho\rangle=k_0$ and  
$$S=\bigoplus_{m\in \alpha+L}\KK\chi^m\neq\{0\}.$$ 
We have $S=\mathfrak{m}_x S$.
The ideal $\mathrm{I}_{\widehat{\rho}}$ is finitely generated. Hence, $S$ is a finitely generated $\KK[\overline{O_\rho}]$-module. 
By Nakayama's lemma there exists an element $u\in \mathfrak{m}_x$ such, that $u\cdot f=f$ for all $f\in S$. If $u\cdot \chi^m=\chi^m$, then $u=1$, but 
$1\notin \mathfrak{m}_x$.

Inductive step. Suppose $k'>k_0$ and for every $k<k'$ the statement is proved. Then
$$\mathrm{I}_{\widehat{\rho}}\cap V_{k'}\subset \mathrm{I}_{\widehat{\rho}}^2\oplus\mathfrak{m}_x\mathrm{I}_{\widehat{\rho}}.$$ But inductive assumption implies that $\mathrm{I}_{\widehat{\rho}}^2\cap V_{k'}\subset W$. If $\mathrm{I}_{\widehat{\rho}}\cap V_{k'}$ is not a subset of $W$, then there is an  $\alpha\in M\setminus L$ such, that $\langle \alpha,n_\rho\rangle=k'$ and
$$S_{k'}=\bigoplus_{m\in \alpha+L}\KK\chi^m\neq\{0\}.$$
From the other hand $S_{k'}\cap \mathrm{I}_{\widehat{\rho}}^2=\{0\}$, hence $S_{k'}\subset \mathfrak{m}_x\mathrm{I}_{\widehat{\rho}}$. Therefore, $S_{k'}= \mathfrak{m}_x S_{k'}$. Again we obtain a contradiction with Nakayama's lemma. 

Thus, we prove, that $\mathrm{I}_{\widehat{\rho}}\subset W$. Hence, $\KK[X]\subset W$. Since the action of $T$ on $X$ is effective, it implies that $\Lambda=M$.

Assume that the face $\widehat{\rho}$ is nowhere saturated. 
Theorem~3.3 from \cite{TY} asserts that, for some element $b$ of Hilbert basis of the monoid $P_{sat}$ there is no any decomposition of the view
$$
b=x_1a_1+\ldots+x_r a_r,\ \text{where} \ x_i\in \mathbb{Z},\ \text{and}\ \forall j>l: x_j\geq 0
$$
It follows from  $\Lambda=M$ that $L= M\cap \rho^{\bot}$. Therefore, $b$ does not belong to the face $\widehat{\rho}$. Let $\langle b,n_\rho\rangle=k$. Then $V_k\cap I_{\widehat{\rho}}=\{0\}$.  Hence, $k$ is not divisible by $k_0$. Therefore, $k_0\neq 1$. From the other hand since $T$ acts on $X$ effectively,  there is such $\widetilde{k}$ that $\widetilde{k}$ is not divisible by  $k_0$ and $V_{\widetilde{k}}\cap\mathrm{I}_{\widehat{\rho}}\neq \{0\}$. Let us assume that $\widetilde{k}$ is a minimal number satisfying this condition. Then $V_{\widetilde{k}}\cap\mathrm{I}_{\widehat{\rho}}^2=\{0\}$. Again Nakayama's lemma implies that the following inclusion is false $V_{\widetilde{k}}\cap\mathrm{I}_{\widehat{\rho}}\subset \mathfrak{m}_x\mathrm{I}_{\widehat{\rho}}$. Hence, $\dim\mathrm{I}_{\widehat{\rho}}/\left(\mathrm{I}_{\widehat{\rho}}^2+\mathfrak{m}_x\mathrm{I}_{\widehat{\rho}}\right)\geq 2$. We obtain a contradiction. 
\end{proof}

\section{Flexible toric varieries}

In this section we prove a criterium for an affine toric variety to be flexible, see Theorem~\ref{gib}. Also we give its reformulations in geometrical and in combinatorial terms, see Corollaries~\ref{geo}, \ref{vt} and \ref{dy}.

\begin{lemma}\label{obr}
Let $X$ be a flexible affine variety. Then the quasiaffine variety $X^{\mathrm{reg}}$ does not admit any nonconstant invertible functions.
\end{lemma}
\begin{proof}
Since $X$ is flexible, $\mathrm{SAut}(X)$ acts on $X^{\mathrm{reg}}$ transitively. An automorphism takes smooth points to smooth points. Hence each $\mathbb{G}_a$-action on $X$ can be restricted to an action on $X^{\mathrm{reg}}$. By \cite{Ar} each $\mathbb{G}_a$-action on a quasiaffine variety corresponds to an LND of algebra of regular functions. But each invertible function belongs to kernels of all LNDs. Therefore, each invertible function $f$ is an $\mathrm{SAut}(X)$-invariant. Hence, sets $\{f=c\}$ are $\mathrm{SAut}(X)$-invariant for all $c\in\mathbb{K}$. If $f$ is nonconstant, this contradicts to flexibility of $X$. 
\end{proof}

\begin{theorem}\label{gib}
Let $X$ be a toric variety and $\sigma$ be the corresponding cone. Let us remove all extremal rays of the cone $\sigma$, that correspond to orbits, consisting of singular points. Let $\gamma$ be the cone generated by all other extremal rays. The variety $X$ is flexible if and only if  the cone $\gamma$ is not contained in any hyperspace.
\end{theorem}
\begin{proof}
Assume $\gamma$ is contained in a hyperspace $H\subset N_\QQ$. Then the cone $\gamma^\vee$ contains a line $H^\bot$. By Lemma~\ref{co} existence of a line in $\gamma^\vee$ implies existence of a nonconstant invertible function on $X^{reg}$. By Lemma~\ref{obr} the variety $X$ is not flexible.

Now assume that the cone $\gamma$ is not contained in any hyperspace. Hence  there are $n$ extremal rays $\rho_1,\ldots,\rho_n$ of $\sigma$ corresponding to smooth orbits such that $\{n_{\rho_1},\ldots,n_{\rho_n}\}$ is a basis of~$N_\QQ$. By Lemma~\ref{eu} there exist Demazure roots $e_1,\ldots,e_n$, where $e_i\in{\mathscr R}_{\rho_i}$, such that all LNDs $\delta_{e_i}$ are well defined. Consider the point $p=(1,1,\ldots,1)\in T\subset X$. For the standard basis
$$\{m_1=(1,0,\ldots,0),\ldots,m_n=(0,\ldots,0,1)\}$$  
of the lattice $M$ we have $\chi^{m_i}(p)=1$. Therefore, for every $m\in M$ we have $\chi^m(p)=1$. Tangent vector to the $\mathbb{G}_a$-orbit
$\mathcal{H}_{e_i} p$ has the view $\left(\delta_{e_i}(\chi^{m_1})(p),\ldots,\delta_{e_i}(\chi^{m_n})(p)\right)$. And
$$
\delta_{e_i}(\chi^{m_j})(p)=\langle m_j,n_{\rho_i}\rangle \chi^{m_j+e_i}(p)=\langle m_j,n_{\rho_i}\rangle.
$$
To prove that tangent vectors in the point $p$ to all orbits $\mathcal{H}_{e_i}p$ are linearly independent, we are to prove that the determinant of the $n\times n$-matrix  
 with elements $\langle m_j,n_{\rho_i}\rangle$ is not equal to zero. This is true since $\{m_1,\ldots, m_n\}$ is a basis of $M_\QQ$, and $\{n_{\rho_1},\ldots,n_{\rho_n}\}$ is a basis of $N_\QQ$. 

We obtain that $p$ is a flexible point of the variety $X$. Since $p$ belongs to the open $T$-orbit, this orbit consists of flexible points.

For each face $\tau$ of the cone $\sigma^\vee$ we can take a primitive, i.e. with relatively prime coordinates, integer vector $n\in N$ such, that $\tau=\sigma^\vee\cap\langle n\rangle^\bot$. This gives the following grading $\KK[X]_i=\oplus_{\langle m,n\rangle=i}\KK \chi_m$. For this grading all negative homogeneous components are zero. It is follows from  \cite[Proposition~3, Corollary~1]{GSh} that the closure of a smooth non open orbit is not $\SAut(X)$-invariant. (Since the paper~\cite{GSh} devoted to normal varieties, in conditions of~\cite[Proposition~3, Corollary~1]{GSh} the variety is assumed to be normal. But in the proof the normality of the variety is not used. So this proposition is true for nonnormal varieties too.) Therefore, for each smooth orbit there exists an element in $\SAut(X)$ such that it takes a point in this orbit to a point in the open orbit. Hence all smooth points of  $X$ are flexible. Theorem \ref{gib} is proved.
\end{proof}

\begin{re}
If a toric variety $X$ is normal, then all orbits of codimension~1 are smooth. Hence, the cone  $\gamma$ coincide with $\sigma$. In this case the assertion of Theorem~\ref{gib} coincides with the result of~\cite[Theorem~2.1]{AKZ}.
\end{re}

Let us give some equivalent reformulations of Theorem~\ref{gib}. The first one gives a criterium of flexibility for toric variety in geometric terms, i.e. there is no combinatoric data of monoid $P$ corresponding to the variety $X$.

\begin{corollary}\label{geo}
A toric variety $X$ is flexible if and only if there is no regular function $f\in\KK[X]$ such, that the set of zeros of $f$ contains only singular points.
\end{corollary}
\begin{proof}
The set of zeros of $f$ contains only singular points if and only if  $f$  is invertible on $X^{reg}$. Existing of an invertible function in $\KK[X^{reg}]$ is equivalent to existing of  an invertible homogeneous function $\chi^m\in\KK[X^{reg}]$, $m\neq 0$.  The inverse to $\chi^m$ element  is~$\chi^{-m}$. So in $\KK[X^{reg}]$ there is a homogeneous invertible element $\chi^m$ if and only if there is a line in $\gamma^\vee$. By Theorem~\ref{gib} the cone $\gamma^\vee$ contains a line if and only if $X$ is not flexible.
\end{proof}

The second reformulation of Theorem~\ref{gib} gives a criterium for a toric variety to be flexible in terms of the weight monoid $P$. We obtain it just by union of Theorem~\ref{gib} and Lemma~\ref{eu}. 
\begin{corollary}\label{vt}
A toric variety $X$ is flexible if and only if  there is no hyperspace in $N_\QQ$ containing all extremal rays $\rho_i$ of the cone  $\sigma$ such, that the face $\widehat{\rho_i}$ is almost saturated. 
\end{corollary}

A bounded part of the monoid $P$ does not affect almost saturation of faces. That is if we have two monoids $P$ and  $P'$ that  differ only in finite number of elements, then they corresponds to the same cone and a face $\tau\preccurlyeq \sigma^\vee$ is almost saturated in $P$, if and only if it is almost saturated in~$P'$. This implies the following assertion.
\begin{corollary}
If the monoid $P$ has only finite number of holes and the cone $\sigma$ is not contained in any hyperspace, then the variety $X$ is flexible. 
\end{corollary}

It is easy to see that for a two-dimensional cone the condition of almost saturation of both faces of codimension one is equivalent to finiteness of number of holes. In arbitrary case finiteness of number of holes is not necessary for flexibility. It is possible to locate infinite number of holes in such a way that they do not effect almost saturating of faces. But all holes should locate "near" nowhere saturated faces of codimension one or "near" faces of codimension not less then 2. 
\begin{lemma}\label{per}
Let $P$ be the monoid, corresponding to the cone $\sigma$. Let $\rho_1,\ldots, \rho_k$ be extremal rays of $\sigma$ and $n_{1},\ldots, n_k$ be primitive vectors on them. For a fixed positive integer $d\in\NN$ we denote
$$
L_i(d)=\{m\in M\cap \sigma^\vee\mid \langle m,n_i\rangle\leq d\}.
$$
Then the following conditions are equivalent:

1) faces $\widehat{\rho_1},\ldots,\widehat{\rho_s}$, $s\leq k$ are almost saturated;

2) there exists a constant $c\in\NN$ such, that for each hole $u$ of the monoid $P$ either $u\in L_t(c)$ for some $t>s$, or $u\in L_i(c)\cap L_j(c)$, where $i,j\leq s$.
\end{lemma}
\begin{proof}
$1\Rightarrow 2$. By Lemma~\ref{mg} there is a saturation point $v$ in $\sigma^\vee$. Let $x_i=\langle v,n_i\rangle$. Let us take $d>x_i$ for all $i$. If for all $i$ we have 
$u\notin L_i(d)$, then $\langle u,n_i\rangle>d>\langle v,n_i\rangle$, that is $u-v\in \sigma^\vee\cap M$. Since $v$ is a saturation point, we obtain $u\in P$. This contradicts to the condition that $u$ is a hole. Therefore, there is $i$ such, that $u\in L_i(d)$. 

Let $i\leq s$. Then there is a saturation point $w$ in the face $\widehat{\rho_i}$. Let us denote $y_j=\langle w,n_j\rangle$. Take $c_i>y_j$ for each $j$. If for all $j\neq i$ we have 
$u\notin L_j(c_i)$, then $\langle u,n_j\rangle>c_i>\langle w,n_j\rangle$. Since $\langle w,n_i\rangle=0$, we obtain $u-w\in \sigma^\vee\cap M$. This contradicts the fact that $w$ is a saturation point and $u$ is a hole. Thus there exists $j\neq i$ such that $u\in L_j(c_i)$.

Put $c=\max\{d,c_1,\ldots, c_s\}$. Then $u\in L_i(c)$ and if $i\leq s$, then $u\in L_j(c)$ for some $j\neq i$.

$2\Rightarrow 1.$ Let us fix $i\leq s$. There is a point $v$ in $\widehat{\rho_i}\cap M$ such that $v\notin L_j(c)$ for all $j\neq i$. Therefore, $v$ is a saturation point. 
\end{proof}

Lemma \ref{per} and Corollary \ref{vt} allows to obtain one more reformulation of Theorem~\ref{gib}, which gives a criterium of flexibility of a toric variety in terms of holes in weight monoid. 
\begin{corollary}\label{dy}
Let $X$ be a toric variety. Denote by $P$ the weight monoid and by $\sigma$ the corresponding cone. Then $X$ is flexible if and only if there is an ordering $\rho_1,\ldots, \rho_k$ of extremal rays of the cone~$\sigma$ such that 

1) rays $\rho_1,\ldots, \rho_n$ are not contained in any hyperspace,

2) there is a positive integer $c$ such, that for every hole $u$ of the monoid $P$ either $u\in L_t(c)$ for some $t>n$, or $u\in L_i(c)\cap L_j(c)$, where $i,j\leq n$.
\end{corollary}

\section{Rigid toric varieties}

In this section we prove a criterium for an affine toric variety $X$ to be rigid. Also we give an explicit description of the automorphism group of rigid affine toric varieties.\begin{theorem}\label{zhes}
An affine toric variety $X$ is rigid if and only if the smooth locus $X^{reg}$ coincides with the open $T$-orbit.
\end{theorem}
\begin{proof}
Smooth locus is open. Hence if we are given by a nonopen $T$-orbit on $X$, then each orbit containing in its closure the given one is smooth. Hence there is an orbit of codimension one consisting of smooth points. By Lemma~\ref{eu} there exists a nonzero LND~$\delta_e$ of $\KK[X]$, i.e. the variety $X$ is not rigid.

Now let $X^{reg}$ coincide with the open $T$-orbit. Then $X^{reg}\cong T$, and hence, $\KK[X^{reg}]$ is generated by invertible functions. Therefore, $\KK[X^{reg}]$ does not admit any nonzero LND, i.e. $X^{reg}$ is rigid. Since each nontrivial $\mathbb{G}_a$-action on $X$ induces a nontrivial $\mathbb{G}_a$-action on $X^{reg}$, the variety $X$ is rigid.
\end{proof}
\begin{re}
Lemma~\ref{eu} gives a criterium of rigidity for a toric variety $X$ in terms of its weight monoid $P$. The variety $X$ is rigid if and only if all faces of codimension one of the cone $\sigma^\vee$ are nowhere saturated. 
\end{re}

The automorphism group of a rigid variety possess a unique maximal torus, which is a normal subgroup of $\mathrm{Aut}(X)$, see. \cite{AG}. Using this we can describe the automorphism group of a rigid affine toric variety.

\begin{theorem}\label{autzhes}
Let  $X$ be a rigid affine toric variety. The automorphism group $\Aut(X)$ is isomorphic to a semidirect product $S(X)\rightthreetimes T$. 
\end{theorem}
\begin{proof}
Let $\psi$ be an automorphism of $\KK[X]$. The torus $T$ is a normal subgroup of $\Aut(X)$, hence the automorphism $\psi$ takes $M$-homogeneous functions to $M$-homogeneous.  Since $\psi$ respects the multiplication, there exists a linear operator $\varphi\colon M_\QQ\rightarrow M_\QQ$ such that $\varphi(M)=M$ and for each $m\in M$ there exists a nonzero constant $\lambda$ such, that $\psi(\chi^m)=\lambda\chi^{\varphi(m)}$. Let us consider the automorphism $\zeta=\psi\circ \alpha(\varphi)^{-1}$. We have $\zeta(\chi^m)=\lambda \chi^m$. Since $\zeta$ is an automorphism, the constants $\lambda$ are organized in such a way that the action of $\zeta$ on $\KK[X]$ coincide with the action of an element of the torus $T$. So we see that $S(X)$  and  $T$ generate $\Aut(X)$.  It easy to see that the intersection of this subgroups is trivial. Since $T$ is a normal subgroup of $\Aut(X)$, we obtain $\Aut(X)\cong S(X)\rightthreetimes T$.
\end{proof}

\begin{ex}
Let us consider the toric variety $X$, corresponding to the monoid $P$ consisting of all integer points $(a,b)$ such that $a\geq 0, b\geq 0$, except all points $(0,2k+1)$ and all points $(2k+1,0)$. This variety is rigid. Indeed both one-dimensional faces are nowhere saturated. 
$$
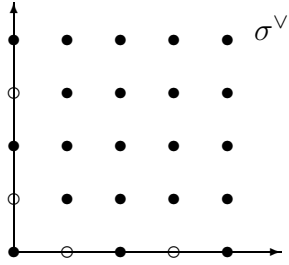
\begin{picture}(100,93)
\put(90,80){$\sigma^\vee$}
\put(0,0){\vector(1,0){100}}
\put(0,0){\vector(0,1){94}}
\put(0,0){\circle*{4}}
\put(0,20){\circle{4}}
\put(0,40){\circle*{4}}
\put(0,60){\circle{4}}
\put(0,80){\circle*{4}}
\put(20,0){\circle{4}}
\put(20,20){\circle*{4}}
\put(20,40){\circle*{4}}
\put(20,60){\circle*{4}}
\put(20,80){\circle*{4}}
\put(40,0){\circle*{4}}
\put(40,20){\circle*{4}}
\put(40,40){\circle*{4}}
\put(40,60){\circle*{4}}
\put(40,80){\circle*{4}}
\put(60,0){\circle{4}}
\put(60,20){\circle*{4}}
\put(60,40){\circle*{4}}
\put(60,60){\circle*{4}}
\put(60,80){\circle*{4}}
\put(80,0){\circle*{4}}
\put(80,20){\circle*{4}}
\put(80,40){\circle*{4}}
\put(80,60){\circle*{4}}
\put(80,80){\circle*{4}}
\end{picture}
$$
The subgroup $S(X)$ is isomorphic to $\ZZ_2$. It consists of the identity automorphism and the automorphism reflecting the cone relative the diagonal. The automorphism group of $X$ is isomorphic to $\ZZ_2\rightthreetimes (\KK^\times)^2$.
\end{ex}

\section{Almost rigid toric varieties}
In this section we investigate one more class of toric varieties, for which it is possible to describe automorphism group. This group is not algebraic, but it is a semidirect product of four subgroups. Each of these four subgroups has an explicit description. 
\begin{definition}
A variety $X$ is called {\it almost rigid}, if it is not rigid and every two LNDs of $\KK[X]$ commute. 
\end{definition}
If $X$ is almost rigid, we denote the commutative group $\SAut(X)$ by $\mathcal{U}(X)$. Lemma~\ref{pzh} implies, that the variety $X$ is almost rigid if and only if kernels of all LNDs of $\KK[X]$ coincide.  
If we are given by an LND $\partial$, we can consider the following commutative subgroup in automorphism group $\mathbb{U}(\partial)=\{\exp (f\partial)\mid f\in \Ker \partial\}.$
Sometimes under the term "almost rigid variety" one mean such a variety $Y$, that all LNDs of $\KK[Y]$ are replicas of a fixed one $\partial$. In this case $\mathcal{U}(Y)=\mathbb{U}(\partial)$. But there exists an almost rigid in our terminology variety $X$ such that for each LND $\partial$ of $\KK[X]$ we have $\mathbb{U}(\partial)\subsetneq \mathcal{U}(X)$, see example ~\ref{spzh}.

\begin{theorem}\label{pz}
A toric variety $X$ is almost rigid if and only if there is a unique extremal ray  $\rho$ of the cone $\sigma$ such that the face $\widehat{\rho}$ is almost saturated.

\end{theorem}
\begin{proof}
If the cone $\sigma^\vee$ has not almost saturated faces of codimension one, then by Theorem~\ref{zhes} the variety $X$ is rigid. 
Assume the cone $\sigma^\vee$ has an almost saturated face $\rho$ of codimension one. By Lemma~\ref{eu} there is a Demazure root  $e\in{\mathscr R}_{\rho}$ such, that the LND $\delta_{e}$ is well defined. Then we have
$$\Ker \delta_{e}=\bigoplus_{m\in P\cap \widehat{\rho}}\KK\chi^m.$$ 
If there are at least two almost saturated faces of codimension 1, then there are two LNDs with different kernels. Therefore, $X$ is not almost rigid. 

Now assume that the cone $\sigma^\vee$ has a unique almost saturated face $\tau$ of codimension one. Let $\partial$ be an LND of $\mathbb{K}[X]$. Then $\partial$ can be decomposed onto a sum of $M$-homogeneous derivations. Vertices of the convex hull of degrees of summands are degrees of $M$-homogeneous LNDs, i.e. Demazure roots. All of them have degree $-1$ with respect to the grading corresponding to the face $\tau$. Hence, $\partial$ also has degree $-1$ with respectt to this grading. Therefore, $\mathrm{Ker}\,\partial$ contains the zero homogeneous component of this grading. By \cite[Principle~11]{Fr} the transcendence degree of $\Ker\,\partial$ over $\KK$ equals to the transcendence degree of $\KK[X]$ minus one. Hence $\Ker\,\partial$ coincides with the zero component. Hence, kernels of all LNDs coincide. That is $X$ is almost rigid.
\end{proof}

Let us describe the automorphism group of an almost rigid toric variety. 
Let $\widehat{\rho}$ be the unique almost saturated face of codimension one of the cone $\sigma^\vee$. It corresponds to the grading 
$$
\KK[X]=\bigoplus_{i\in\mathbb{Z}}\KK[X]_i, \ \ \ \ \text{where}\ \  \ \KK[X]_i=\bigoplus_{\langle m,n_\rho \rangle=i}\chi^m.
$$
Let us fix an invertible function $h\in\KK[X]^\times$. We define an automorphism $\varphi_h\colon \KK[X]\rightarrow \KK[X]$ by the following rule. If $f\in\KK[X]_k$ we have $\varphi_h(f)=h^kf$.  All invertible functions are in $\KK[X]_0$. Therefore, 
$H(X)=\{\varphi_h\mid h\in \KK[X]^\times\}$ is a commutative group.  The group $H(X)$ intersects with $T$ by the following one-dimensional torus $\Lambda=\{\varphi_c\mid c\in \KK^\times\}$. 
If we decompose $T$ onto a direct product $T=\Lambda \times \Omega$, where $\Omega\cong (\KK^\times)^{n-1}$, then the group generated by $H(X)$ and $T$ is a semiditect product $\Omega\rightthreetimes H(X)$.

\begin{theorem}\label{apz}
The automorphism group of an almost rigid toric variety $X$ has the following view:
$$\left(S(X)\rightthreetimes\Omega\right)\rightthreetimes \left(H(X)\rightthreetimes \mathcal{U}\left(X\right)\right).$$
\end{theorem}
To prove this Theorem firstly we consider a particular case.
\begin{lemma}\label{pos}
Let $Y=\KK\times(\KK^\times)^{n-1}$, $\KK[Y]=\KK[x_1,x_2,x_2^{-1}, \ldots, x_n,x_n^{-1}]$. Then the automorphism group $\Aut(Y)$  is isomorphic to a semidirect product 
$$ \left(\GL_{n-1}(\ZZ)\rightthreetimes\Omega\right)\rightthreetimes \left(H(Y)\rightthreetimes \mathbb{U}\left(\frac{\partial}{\partial x_1}\right)\right).$$
\end{lemma}
\begin{proof}
Invertible elements of the algebra $A=\KK[x_1,x_2,x_2^{-1}, \ldots, x_n,x_n^{-1}]$ have the view $cx_2^{l_2}\ldots x_n^{l_n}$, where $l_i\in \ZZ$. The subalgebra generated by all invertible functions is 
$B=\KK[x_2,x_2^{-1}, \ldots, x_n,x_n^{-1}]$. This subalgebra is invariant under every automorphism.
It is known that the automorphism group of $B$ is isomorphic to $\GL_{n-1}(\ZZ)\rightthreetimes \Omega$, where $\Omega\cong (\KK^\times)^{n-1}$ is the torus, acting by multiplication of $x_2,\ldots, x_n$ by nonzero constants. Thus, the restriction of the action the automorphism group to $B$ gives a homomorphism $\Psi\colon \Aut(A)\rightarrow \GL_{n-1}(\ZZ)\rightthreetimes \Omega$. The homomorphism $\Psi$ is surjective. Moreover, there is a natural subgroup in $\Aut(A)$ preserving $x_1$, which maps on $\Aut(B)$ isomorphicly. Let us prove, that $\Ker\Psi=H(Y)\rightthreetimes\mathbb{U}(\frac{\partial}{\partial x_1})$. 
Take $\alpha\in\Ker\Psi$. Let $\mathbb{L}=\overline{\KK(x_2,\ldots, x_n)}$ be the algebraic closure of the field obtained by adjoining $x_2,\ldots,x_n$ to $\KK$. Then $\alpha$ induces an automorphism of $\mathbb{L}[x_1]$. Automorphism of the polynomial algebra in one variable $\mathbb{L}[x_1]$ have the view $x_1\mapsto ax_1+b$, where $a\in\mathbb{L}^\times$, $b\in \mathbb{L}$. This formula gives an autoharphism of $A$, if and only if $a$ is an invertible element of $A$. That is $a=cx_2^{l_2}\ldots x_n^{l_n}$, $l_i\in\ZZ$.
\end{proof}

\begin{proof}[Proof of Theorem~\ref{apz}]
By Lemma~\ref{co} we have 
$$\KK[X^{reg}]=\bigoplus_{m\in M\cap\gamma^\vee}\KK\chi^m.$$ 
Since $\sigma^{\vee}$ has a unique almost saturated face of codimension one, the cone $\gamma^{\vee}$ is a half-space. So, the algebra $\KK[X^{reg}]$ is isomorphic to $A=\KK[x_1,x_2,x_2^{-1}, \ldots, x_n,x_n^{-1}]$. Since the restriction of each automorphism $X$ to $X^{reg}$ gives an automorphism of $X^{reg}$ and different automorphisms of $X$ give different automorphisms of  $X^{reg}$, we have the following embedding $\Aut(X)\hookrightarrow\Aut(X^{reg})\cong\Aut(A)$. 
The subalgebra of $A$ generated by all invertible functions is 
$B=\KK[x_2,x_2^{-1}, \ldots, x_n,x_n^{-1}]$. It is invariant under each automorphism. We obtain a homomorphism $\rho\colon \Aut(X)\rightarrow \Aut(B)$. The automorphism group of $B$ is isomorphic to $\GL_{n-1}(\ZZ)\rightthreetimes\, \Omega$, where $\Omega\cong (\KK^\times)^{n-1}$ is a torus, acting by multiplication of $x_2,\ldots, x_n$ by nonzero constants. Therefore, we have a homomorphism $\eta\colon \mathrm{Im}\,\rho\rightarrow\GL_{n-1}(\ZZ)$. By definition the group $S(X)$ is the image of the homomorphism~$\eta$. Since the image of the homomorphism~$\rho$ contains $\Omega$, we have $\mathrm{Im}\,\rho=S(X)\rightthreetimes\, \Omega$.
The kernel of the homomorphism $\rho$ is the following intersection $\Aut(X)\cap \left(H(Y)\rightthreetimes \mathbb{U}(\frac{\partial}{\partial x_1})\right)$. So we obtain a homomorphism $\eta\colon \mathrm{Ker}\,\rho\rightarrow H(Y)$. We have $\mathrm{Ker}\,\eta=\Aut(X)\cap \mathbb{U}(\frac{\partial}{\partial x_1})$. Since all elements of  $\mathbb{U}(\frac{\partial}{\partial x_1})$ are exponents of some LNDs and there are no any more LNDs of $\KK[X^{reg}]$, we have $\mathrm{Ker}\,\eta=\mathcal{U}(X)$.  And $\mathrm{Im}\,\eta=\Aut(X)\cap H(Y)=H(X)$.

\end{proof}

\begin{ex}\label{sl}
Let us consider the toric variety $X$, corresponding to the monoid $P$, consisting of all integer points $(a,b)$ such, that $a\geq 0, b\geq 0$, exept all points $(2k+1,0)$. This variety is almost rigid, since the face $\mathbb{Q}_{\geq 0}(0,1)$ of $\sigma^\vee$ is almost saturated, and the face $\mathbb{Q}_{\geq 0}(1,0)$ is nowhere saturated. 
$$
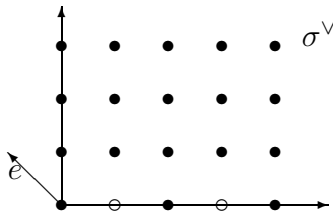
\begin{picture}(120,70)
\put(110,60){$\sigma^\vee$}
\put(20,0){\vector(1,0){100}}
\put(20,0){\vector(0,1){75}}
\put(20,0){\circle*{4}}
\put(20,20){\circle*{4}}
\put(20,40){\circle*{4}}
\put(20,60){\circle*{4}}
\put(40,0){\circle{4}}
\put(40,20){\circle*{4}}
\put(40,40){\circle*{4}}
\put(40,60){\circle*{4}}
\put(60,0){\circle*{4}}
\put(60,20){\circle*{4}}
\put(60,40){\circle*{4}}
\put(60,60){\circle*{4}}
\put(80,0){\circle{4}}
\put(80,20){\circle*{4}}
\put(80,40){\circle*{4}}
\put(80,60){\circle*{4}}
\put(100,0){\circle*{4}}
\put(100,20){\circle*{4}}
\put(100,40){\circle*{4}}
\put(100,60){\circle*{4}}
\put(20,0){\vector(-1,1){20}}
\put(0,10){$e$}
\end{picture}
$$
The subgroup $S(X)$ is trivial for this variety. It is easy to see, that the subgroup $\mathcal{U}(X)$ has the view $\mathbb{U}(\delta_e)$, where $e=(-1,1)$ . The automorphism group of $X$ is isomorphic to $T\rightthreetimes \mathbb{U}(\delta_e)$.
\end{ex}

\begin{ex}\label{spzh}
Let us take the monoid $P$ from the preveos example and remove the point $(0,1)$. This removing does not change existing or absence of saturation points on faces. Hence, the variety $X$ is almost rigid.
$$
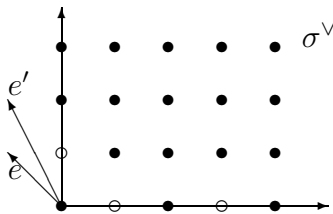
\begin{picture}(120,70)
\put(110,60){$\sigma^\vee$}
\put(20,0){\vector(1,0){100}}
\put(20,0){\vector(0,1){75}}
\put(20,0){\circle*{4}}
\put(20,20){\circle{4}}
\put(20,40){\circle*{4}}
\put(20,60){\circle*{4}}
\put(40,0){\circle{4}}
\put(40,20){\circle*{4}}
\put(40,40){\circle*{4}}
\put(40,60){\circle*{4}}
\put(60,0){\circle*{4}}
\put(60,20){\circle*{4}}
\put(60,40){\circle*{4}}
\put(60,60){\circle*{4}}
\put(80,0){\circle{4}}
\put(80,20){\circle*{4}}
\put(80,40){\circle*{4}}
\put(80,60){\circle*{4}}
\put(100,0){\circle*{4}}
\put(100,20){\circle*{4}}
\put(100,40){\circle*{4}}
\put(100,60){\circle*{4}}
\put(20,0){\vector(-1,1){20}}
\put(0,10){$e$}
\put(20,0){\vector(-1,2){20}}
\put(0,42){$e'$}
\end{picture}
$$
Subgroups $S(X)$, $\mathcal{U}(X)$, and $\mathrm{Aut}(X)$ are the same as in the previous example. But now not all LNDs are replicas of~$\delta_e$. Indeed, for example, for $e'=(-1,2)$ we have $\delta_{e'}=\chi^{(0,1)}\partial_e$. But $\chi^{(0,1)}$ is not a regular function on $X$. So, there is no LND $\delta$ such that $\mathcal{U}(X)$ coincides with~$\mathbb{U}(\delta)$. 
\end{ex}


\end{document}